\documentclass[12pt]{article}

\usepackage[margin=1in]{geometry}

\usepackage{amsmath, amssymb, amsthm, amssymb }
\usepackage{amsfonts}
\usepackage{graphicx}
\usepackage[frame,ps,matrix,arrow,curve,rotate]{xy}
\usepackage{enumerate}
\usepackage{hyperref}
\usepackage{enumitem}
\usepackage{MnSymbol}
\usepackage{thm-restate}

\newtheorem{theorem}{Theorem}[section]
\newtheorem{lemma}[theorem]{Lemma}

\newtheorem{corollary}[theorem]{Corollary}

\theoremstyle{definition}

\newtheorem{example}[theorem]{Example}

\theoremstyle{remark}
\newtheorem{remark}[theorem]{Remark}

\newcommand{\mc}[1]{\mathcal{#1}}

\newcommand{\res}{\!\!\upharpoonright}

\newcommand{\bfSigma}{\boldsymbol{\Sigma}}

\newcommand{\bfPi}{\mathbf{\Pi}}

\newcommand{\ol}{\overline}

\newcommand{\forces}{\Vdash}

\DeclareMathOperator{\tp}{tp}
\DeclareMathOperator{\Mod}{Mod}

\makeatletter
\newcommand\mathcircled[1]{%
	\mathpalette\@mathcircled{#1}%
}
\newcommand\@mathcircled[2]{%
	\tikz[baseline=(math.base)] \node[draw,circle,inner sep=1pt] (math) {$\m@th#1#2$};%
}
\makeatother

\setenumerate[1]{label={(\arabic*)}}

\begin{document}
	
	\title{Infinitary Logic Has No Expressive Efficiency Over Finitary Logic}
	\author{Matthew Harrison-Trainor \and Miles Kretschmer\thanks{Miles Kretschmer was supported by an REU at the University of Michigan funded by NSF grant 2003712.}}
	
	\maketitle
	
	\begin{abstract}
		We can measure the complexity of a logical formula by counting the number of alternations between existential and universal quantifiers. Suppose that an elementary first-order formula $\varphi$ (in $\mc{L}_{\omega,\omega}$) is equivalent to a formula of the infinitary language $\mc{L}_{\infty,\omega}$ with $n$ alternations of quantifiers. We prove that $\varphi$ is equivalent to a finitary formula with $n$ alternations of quantifiers. Thus using infinitary logic does not allow us to express a finitary formula in a simpler way.
	\end{abstract}
	
	\section{Introduction}\label{section:introduction}
	
	This paper is about the relationship between finitary elementary first-order logic ($\mc{L}_{\omega,\omega}$) and the infinitary logics (such as $\mc{L}_{\infty,\omega}$) which extend it by allowing conjunctions and disjunctions of infinite sets of formulas. These infinitary logics are more expressive than the finitary logic, but lose compactness. For example, there are classes such as connected graphs and torsion groups which cannot be axiomatized in the finitary logic (as shown by a simple compactness argument) but which can be axiomatized in the infinitary logic.
	
	A natural way of measuring the complexity of a formula is by putting it in normal form and counting the number of alternations of quantifiers. An $\exists_n$ formula is a formula which begins with a block of existential quantifiers, and has $n$ alternating blocks of existential and universal quantifiers. Similarly, a $\forall_n$ formula has $n$ alternating blocks of existential and universal quantifiers, beginning with a block of universal quantifiers. In the infinitary languages, we count alternations of quantifiers, but we do \textit{not} count infinitary conjunctions and disjunctions. (This differs from the standard way of counting quantifiers in computable structure theory, where formulas are classified as $\Sigma_n$ or $\Pi_n$; a $\Sigma_n$ formula is $\exists_n$ but not necessarily vice versa.\footnote{Our theorems remain true if one replaces $\exists_n$ by $\Sigma_n$ and $\forall_n$ by $\Pi_n$, so we get stronger theorems by not counting infinitary conjunctions and disjunctions.})
	
	Now suppose that we have a finitary formula $\varphi(\ol{x})$ which is equivalent to an infinitary existential ($\exists_1$) formula, say
	\[ \psi(\ol{x}) = \bigdoublewedge_{i} \bigdoublevee_{j} \exists \ol{y}_{i,j} \; \theta_{i,j}(\ol{x},\ol{y}_{i,j}).\]
	Now one can check that $\psi(\ol{x})$ (and hence $\varphi(\ol{x})$) has the property of being preserved in extensions: if $\mc{A} \subseteq \mc{B}$, and $\mc{A} \models \psi(\ol{a})$, then $\mc{B} \models \psi(\ol{a})$. Since $\varphi(\ol{x})$ is preserved by extensions, a standard preservation result says that $\varphi(\ol{x})$ is itself equivalent to a finitary existential ($\exists_1$) formula. In general, one proves that if a finitary formula is equivalent to an infinitary existential formula, it is equivalent to a finitary existential formula.
	
	If $\varphi(\ol{x})$ is instead a finitary formula which is equivalent to an infinitary $\forall_2$ formula $\psi(\ol{x})$, a similar argument works to show that $\varphi(\ol{x})$ is equivalent to a finitary $\forall_2$ formula. Instead of preservation in extensions, we use the fact that $\varphi(\ol{x})$ is equivalent to a finitary $\forall_2$ formula if and only $\varphi$ is preserved under unions of chains of models: Whenever $\mc{A}_0\subseteq\mc{A}_1\subseteq\mc{A}_2\subseteq\dots \subseteq \mc{A}_{\omega} = \bigcup_{n<\omega}\mc{A}_n$ is a chain of models, with $\mc{A}_n\models\varphi(\ol{a})$ for all $n$, then $\mc{A}_{\omega}\models\varphi(\ol{a})$. In this case, $\varphi(\ol{x})$ is preserved under unions of chains of models because it is equivalent to the infinitary $\forall_2$ formula $\psi(\ol{x})$, which one can check is preserved under unions of chains of models.
	
	This sort of argument breaks down at the $\forall_4$/$\exists_4$ case, as there is no longer a semantic test for finitary $\forall_4$/$\exists_4$ formulas which is satisfied by infinitary $\forall_4$/$\exists_4$ formulas. Nevertheless, the main result of this paper is that this result is true in general.
	
	\begin{restatable}{theorem}{interpolationtheorem}\label{theorem:interpolation}
		Let $T$ be a finitary ($\mc{L}_{\omega,\omega}$) theory. Let $\psi$ be an infinitary ($\mc{L}_{\infty,\omega}$) $\exists_n$ (resp. $\forall_n$) formula which is equivalent to a finitary formula $\varphi$ in all models of $T$. Then, $\psi$ and $\varphi$ are equivalent to a finitary $\exists_n$ (resp. $\forall_n$) formula in all models of $T$.
	\end{restatable}
	
	\noindent In spite of their greater expressive power, infinitary languages cannot define relations already definable in finitary first order logic with any greater efficiency (as measured by quantifier complexity). One can also view this as computing an intersection: the properties which are expressible by both a finitary formula and an infinitary $\exists_n$ formula are exactly the properties expressible by a finitary $\exists_n$ formula. (One might compare the spirit of this result to Louveau's theorem \cite{Louveau80}: If $C\subset \omega^\omega$ is $\Delta^1_1$ (hyperarithmetic) and $\bfSigma^0_\alpha$, for some $\alpha<\omega_1^{\mathrm{CK}}$, then $C$ is $\Sigma^0_\alpha(x)$ for some hyparithmetic $x$.)
	
	The general technique we use is a notion of forcing where the conditions are elementary extensions of a given structure. This can also be viewed as a form of Robinson's model-theoretic forcing \cite{Robinson71}. One can perhaps view the forcing as fixing a deficiency of infinitary formulas; while the truth of an infinitary formula $\varphi$ may not be preserved by elementary extensions, forcing will be: if $\mc{A} \forces^* \varphi$ and $\mc{B} \succeq \mc{A}$, then we will have that $\mc{B} \forces^* \varphi$. If $\varphi$ is equivalent to a finitary formula, then the truth of $\varphi$ is already preserved under elementary extensions, and we will be able to show that $\mc{A} \models \varphi$ if and only if $\mc{A} \forces^* \varphi$.
	
	The definability of the forcing relation will allow us to characterize those infinitary formulas which do transfer across elementary extensions. We will show that these are the infinitary formulas built up using conjunctions and disjunctions of finitary formulas.
	
	\begin{restatable}{theorem}{preservationtheorem}\label{theorem:preservation}
		Let $\psi(\ol{x})$ be an infinitary $(\mc{L}_{\infty,\omega})$ formula and let $T$ be a finitary ($\mc{L}_{\omega,\omega}$) theory. The following are equivalent:
		\begin{enumerate}
			\item  Given $\mc{A} \preceq \mc{B}$, models of $T$, $\mc{A} \models \varphi(\ol{a})$  if and only if $\mc{B} \models \varphi(\ol{a})$.
			\item $\psi(\ol{x})$ is equivalent in all models of $T$ to an $\mc{L}_{\infty,\omega}$ formula of the form
			\[ \bigdoublevee_\alpha \bigdoublewedge_\beta \theta_{\alpha,\beta}(\ol{x}) \]
			where each $\theta_{\alpha,\beta}$ is a finitary formula.
		\end{enumerate}
		Moreover, if these conditions hold and $\psi$ is a $\forall_n$ (resp. $\exists_n$) formula, then we may take each $\theta_{\alpha,\beta}$ to be $\forall_n$ (resp. $\exists_n$).
	\end{restatable}
	
	We call a formula as in (2) an \textit{elementary formula}. One particular consequence is that if an infinitary formula $\varphi$ is preserved under elementary extensions, and is $\exists_{n+1}$, then whenever $\mc{A} \models \varphi$ and $\mc{A} \preceq_{n} \mc{B}$, $\mc{B} \models \varphi$. 
	
	Note that even if $\psi(\ol{x})$ is in $\mc{L}_{\omega_1,\omega}$, the formula $\bigdoublevee_\alpha \bigdoublewedge_\beta \theta_{\alpha,\beta}(\ol{x})$ might not be. However, we show that one can find a formula in $\mc{L}_{\omega_1,\omega}$ witnessing that $\psi(\ol{x})$ is preserved upwards and downwards by elementary extensions. In the following theorem, we say that an infinitary formula is \textit{quantifier-free over finitary formulas} if it can be built by repeatedly taking (infinitary) conjunctions, disjunctions, and negations of finitary formulas.
	
	\begin{restatable}{theorem}{preservationtheoremctble}\label{theorem:preservationctble}
		Let $\psi(\ol{x})$ be an infinitary $\mc{L}_{\omega_1,\omega}$ formula, and $T$ be a countable finitary theory. The following are equivalent:
		\begin{enumerate}
			\item  Given $\mc{A} \preceq \mc{B}$, models of $T$, and $\ol{a} \in \mc{A}$, $\mc{A} \models \psi(\ol{a})$  if and only if $\mc{B} \models \psi(\ol{a})$.
			\item $\psi(\ol{x})$ is equivalent in all models of $T$ to an $\mc{L}_{\omega_1,\omega}$ formula which is quantifier-free over finitary formulas.
		\end{enumerate}
		Moreover, if these conditions hold and $\psi$ is a $\forall_n$ formula (or an $\exists_n$ formula), then $\psi(\ol{x})$ is equivalent in all models of $T$ to an $\mc{L}_{\omega_1,\omega}$ formula which is quantifier-free over finitary $\exists_n$/$\forall_n$ formulas.
	\end{restatable}
	
	Note that if $\psi$ is $\forall_n$, the formula to which it is equivalent could involve both $\exists_n$ and $\forall_n$ formulas. We leave open whether this generalizes to $\mc{L}_{\kappa,\omega}$ for $\kappa > \omega_1$.
	
	\section{Preliminaries}
	
	\subsection{\texorpdfstring{The Infinitary Languages $\mc{L}_{\kappa,\omega}$ and $\mc{L}_{\infty,\omega}$}{The Infinitary Languages}}
	
	Fix a signature $\tau$ and an infinite cardinal $\kappa$. The language $\mc{L}_{\kappa,\omega}(\tau)$, which we will habitually denote $\mc{L}_{\kappa,\omega}$, will be the language which allows infinite conjunctions and disjunctions of size $< \kappa$. We define the language $\mc{L}_{\kappa,\omega}$ to be the smallest set of formulas with the following properties.
	\begin{enumerate}
		\item If $\psi$ is an atomic formula of $\mc{L}_{\omega,\omega}$, $\psi\in \mc{L}_{\kappa,\omega}$.
		\item If $\phi\in\mc{L}_{\kappa,\omega}$, then $\neg\phi\in\mc{L}_{\kappa,\omega}$.
		\item If $\phi \in \mc{L}_{\kappa,\omega}$, then $\forall x \phi \in\mc{L}_{\kappa,\omega}$ and  $\exists x\phi \in\mc{L}_{\kappa,\omega}$.
		\item If $\Phi\subset\mc{L}_{\kappa,\omega}$ is such that $|\Phi|<\kappa$ and only finitely many variables occur freely among elements of $\Phi$, then $\bigdoublevee\limits_{\phi\in\Phi}\phi\in\mc{L}_{\kappa,\omega}$ and $\bigdoublewedge\limits_{\phi\in\Phi}\phi\in\mc{L}_{\kappa,\omega}$.
	\end{enumerate}
	Note that $\mc{L}_{\omega,\omega}$ is just the standard finitary elementary first-order logic. We say that a formula is in $\mc{L}_{\infty,\omega}$ if it is a formula of $\mc{L}_{\kappa,\omega}$ for some $\kappa$. We often refer to the formulas of $\mc{L}_{\omega,\omega}$ as finitary formulas and to those of $\mc{L}_{\infty,\omega}$ as infinitary formulas.
	
	A fragment $\mathbb{A}$ of $\mc{L}_{\kappa,\omega}$ is a set of formulas of $\mc{L}_{\kappa,\omega}$ with the following properties.
	\begin{enumerate}
		\item If $\psi\in\mathbb{A}$, $\neg\psi\in\mathbb{A}$.
		\item If $\psi\in\mathbb{A}$, every subformula of $\psi$ is in $\mathbb{A}$.
	\end{enumerate}
	Starting with any formula $\psi\in\mc{L}_{\kappa^+,\omega}$ and closing under negations and subformulas, we obtain a fragment $\mathbb{A}$ containing $\psi$, of cardinality at most $\kappa$.

	\subsection{\texorpdfstring{$\forall_n$ and $\exists_n$ Formulas}{Counting Quantifiers}}
	
	In order to count quantifier alternations in infinitary formulas, we will define a hierarchy of classes of formulas, ranked by natural numbers. For each $n \in \mathbb{N}$, we define the classes $\forall_n$, $\exists_n$ of formulas of $\mc{L}_{\infty,\omega}$ as follows.
	\begin{enumerate}
		\item If $\psi$ is atomic, then for all $n$, $\psi\in \forall_n$, and $\psi\in\exists_n$.
		\item If $\psi = \neg \phi$, then $\psi\in\exists_n$, (respectively, $\psi\in\forall_n$) if $\phi\in \forall_n$ (respectively, $\phi\in\exists_n$).
		\item If $\psi  = \bigdoublevee\limits_{\phi\in\Phi}\phi$ or $\psi = \bigdoublewedge\limits_{\phi\in\Phi}\phi$, then $\psi\in\exists_n$ (respectively, $\psi\in\forall_n$) if for every $\phi\in\Phi$, $\phi\in\exists_n$ (respectively, $\phi\in\forall_n$).
		\item If $\psi = \exists\ol{y}\phi(\ol{y})$, $\psi \in \exists_n$ if $\phi\in\exists_n$ and $n \geq 1$, and $\psi\in \exists_{n+1}$ if $\phi\in\forall_n$.
		\item If $\psi = \forall\ol{y}\phi(\ol{y})$, $\psi \in \forall_n$ if $\phi\in\forall_n$ and $n \geq 1$, and $\psi\in \forall_{n+1}$ if $\phi\in\exists_n$.
	\end{enumerate}
	For example, $\exists_0 = \forall_0$ consists of quantifier free formulas. We call the $\exists_1$ formulas \textit{existential formulas} and the $\forall_1$ formulas \textit{universal formulas}.
	
	For finitary formulas, we count quantifier alternations in the same way, though of course in (3) the conjunctions and disjunctions are finitary. We say that $\mc{A}$ is an $n$-elementary substructure of $\mc{B}$, and write $\mc{A} \preceq_n \mc{B}$, if $\mc{A} \subseteq \mc{B}$ and for any finitary $\exists_n$ or $\forall_n$ formula $\phi(\ol{x})$ and $\ol{a} \in \mc{A}$, $\mc{A} \models \phi(\ol{a})$ if and only if $\mc{B} \models \phi(\ol{a})$. Thus,  $\mc{A}\preceq_0\mc{B}$ just means that $\mc{A}\subseteq \mc{B}$, and $\mc{A}\preceq \mc{B}$ if and only if $\mc{A}\preceq_n\mc{B}$ for all $n<\omega$. We write $\mc{A} \equiv \mc{B}$ (or $(\mc{A},\ol{a}) \equiv (\mc{B},\ol{b})$) for elementary equivalence.
	
	For infinitary formulas, we note that this way of counting quantifiers differs from the standard way of counting quantifiers for $\mc{L}_{\omega_1,\omega}$ formulas in computable structure theory. That is, an $\exists_n$ formula is not necessarily $\Sigma_n$ and a $\forall_n$ formula is not necessarily $\Pi_n$ (though every $\Sigma_n$ formula is $\exists_n$ and every $\Pi_n$ formula is $\forall_n$). The difference is that for $\Sigma_n$ and $\Pi_n$ formulas, we count infinite disjunctions the same as existential quantifiers, and infinite conjunctions the same as universal quantifiers, whereas for $\exists_n$ and $\forall_n$ formulas we do not count infinite conjunctions and disjunctions at all. So, for example, a formula of the form $\bigdoublewedge_i \exists x \theta_i(x)$ is $\exists_1$ but not $\Sigma_1$.
	
	There are natural reasons to consider both forms of counting. The classes $\Sigma_n$ and $\Pi_n$ have descriptive-set-theoretic meaning via Vaught's version of the Lopez-Escobar theorem \cite{Vaught75,LopezEscobar}: An invariant set $\mathbb{K}$ of structures is $\bfSigma^0_n$ if and only if it is defined by a $\Sigma_n$ formula. Moreover, many connections between definability and computability-theoretic properties work with $\Sigma_n$ and $\Pi_n$ formulas. 
	
	On the other hand, Malitz \cite{Malitz69} showed that a formula of $\mc{L}_{\omega_1,\omega}$ is preserved by substructures if and only if it is universal ($\forall_1$). This is the sort of behaviour we will consider in this paper, and so the classes $\forall_n$ and $\exists_n$ are the appropriate ones to use. We note that because every $\Pi_n$ formula is also $\forall_n$ (and there is no difference for finitary formulas), our results about $\forall_n$ formulas apply to $\Pi_n$ formulas as well. Stating our results for $\forall_n$ formulas is simply their strongest form.
	
	\section{Forcing with Elementary Extensions}\label{sec:forcing}
	
	In attempting to iteratively construct models of infinitary sentences using elementary chains, one is faced with the obstruction that infinitary formulas are not preserved by elementary extensions. In fact, infinitary sentences can be very unstable with respect to elementary extensions.
	\begin{theorem}\label{theorem:alternation}
		There is a sentence $\psi\in\mc{L}_{\omega_1,\omega}$ and a structure $\mc{A}$ such that for any $\mc{B}\succeq \mc{A}$ with $\mc{B}\models\psi$, there is a $\mc{C}\succ \mc{B}$ with $\mc{C}\models\neg\psi$, and for any $\mc{B}\succeq\mc{A}$ with $\mc{B}\models\neg\psi$, there is a $\mc{C}\succ\mc{B}$ with $\mc{C}\models \psi$.
	\end{theorem}
	\begin{proof}
		Let $\tau$ consist of a unary relation symbol $Q$, a binary relation symbol $R$, and a unary relation $P_n$ for each natural number $n$. The structures we will be interested in will be the disjoint unions of certain building blocks. Think of elements satisfying $Q$ as the roots of a block, and $R$ attaching a number of other elements to the root. A ``standard block" consists of a single element $a$ satisfying $Q$, and countably many elements $b_0,b_1,b_2,\dots$ not satisfying $Q$. We let $(a,b_0),(a,b_1),\dots\in R$. For each $n$, $b_n$ satisfies $P_n$, but not $P_m$ for $m\neq n$. A ``non-standard block" consists of the same elements but, in addition, one or more ``non-standard elements" $b_*$ that do not satisfy $P_n$ for any $n$.
		
		Let $\mc{A}$ be the disjoint union of countably copies of the standard block. Any elementary extension $\mc{B}$ of $\mc{A}$ is elementarily equivalent to $\mc{A}$, and so consists of the disjoint union of infinitely many blocks which are either standard, or contain non-standard elements. 
		
		Let
		\[\psi = \forall x \left(Q(x)\rightarrow \exists y\left(R(x,y)\wedge\bigdoublewedge_n \neg P_n(y)\right)\right)\]
		The sentence $\psi$ says that every element satisfying $Q$ has an associated non-standard element, and so belongs to a non-standard block.
		
		Let $\mc{B}$ be an elementary extension of $\mc{A}$. If $\mc{B}\models\psi$, we obtain $\mc{C}\succ\mc{B}$ by adding a single standard block, in which case $\mc{C}\models\neg\psi$. If $\mc{B}\models\neg\psi$, we obtain $\mc{C} \succ \mc{B}$ by adding a single non-standard element to each standard block, in which case $\mc{C}\models\psi$. 
	\end{proof}
	
	To solve the problem that this phenomenon raises for constructing models, we define relations between structures and infinitary sentences that keep track of our ability to make formulas true in further elementary extensions. These relations can be thought of as a notion of forcing, where as forcing conditions we use structures, ordered by elementary extension. This is related to the approach taken by Robinson in \cite{Robinson71}, and can be viewed as an extension of those methods to infinitary languages, where all extensions considered are elementary.
	
	As usual, there is a strong forcing and a weak forcing. The strong forcing is useful for defining genericity, while the weak notion is required for the definability of forcing.
	
	\subsection{The Strong Forcing Relation}
	
	Given a structure $\mc{A}$, $\psi(\ol{x}) \in\mathcal{L}_{\infty,\omega}$, and $\ol{a}\in A$, we define the strong forcing relation $\mc{A}\forces\psi(\ol{a})$ by the following recursive clauses.
	\begin{enumerate}
		\item If $\psi$ is atomic, $\mc{A}\forces \psi(\ol{a})$ if and only if $\mc{A}\models \psi(\ol{a})$.
		\item If $\psi(\ol{x}) = \neg\phi(\ol{x})$, $\mc{A}\forces\psi(\ol{a})$ if and only if for every $\mc{B}\succeq \mc{A}$, $\mc{B}\not\forces \phi(\ol{a})$.
		\item If $\psi(\ol{x}) = \bigdoublevee\limits_{\phi\in \Phi}\phi(\ol{x})$, $\mc{A}\forces\psi(\ol{a})$ if and only if for some $\phi\in\Phi$, $\mc{A}\forces\phi(\ol{a})$.
		\item If $\psi(\ol{x}) = \bigdoublewedge\limits_{\phi \in\Phi}\phi(\ol{x})$, $\mc{A}\forces \psi(\ol{a})$ if and only if for every $\mc{B}\succeq\mc{A}$, and $\phi\in \Phi$, there is a $\mc{C}\succeq\mc{B}$ such that $\mc{C}\forces \phi(\ol{a})$.
		\item If $\psi(\ol{x}) = \exists\ol{y}\phi(\ol{xy})$, $\mc{A}\forces \psi(\ol{a})$ if and only if for some $\ol{b}\in \mc{A}$, $\mc{A}\forces \phi(\ol{ab})$.
		\item If $\psi(\ol{x}) = \forall\ol{y}\phi(\ol{xy})$, $\mc{A}\forces\psi(\ol{a})$ if and only if for every $\mc{B}\succeq\mc{A}$ and $\ol{b}\in \mc{B}$, there is a $\mc{C}\succeq \mc{B}$ such that $\mc{C}\forces \phi(\ol{ab})$.
	\end{enumerate}
	This diverges from the definition of the satisfaction relation in clauses (2), (4) and (6). For finitary formulas, this makes no difference.
	
	\begin{lemma}\label{lemma:finitary-forcing-is-truth}
		If $\psi$ is finitary, $\mc{A}\forces \psi(\ol{a})$ if and only if $\mc{A}\models \psi(\ol{a})$.
	\end{lemma}
	
	\begin{proof}
		We will prove this by induction on the complexity of $\psi$. All cases except those covered by clauses (2), (4) and (6) are identical to the satisfaction relation. For clause (2), let $\psi(\ol{x}) = \neg\phi(\ol{x})$. If $\phi$ is finitary, then $\mc{A}\forces \neg\phi(\ol{a})$ if and only if for every $\mc{B}\succeq \mc{A}$, $\mc{B}\not\forces \phi(\ol{a})$. Appealing to induction, this is true if and only if for every $\mc{B}\succeq\mc{A}$, $\mc{B}\models\neg \phi(\ol{a})$, which is true if and only if $\mc{A} \models\neg\phi(\ol{a})$.
		
		For clause (4), let $\psi = \bigwedge_{\phi\in\Phi}\phi(\ol{x})$, where $\Phi$ is finite.  $\mc{A}\forces \psi(\ol{a})$ if and only if for each $\phi\in\Phi$ and $\mc{B}\succeq\mc{A}$, there is a $\mc{C}\succeq\mc{B}$ such that $\mc{C}\forces \phi(\ol{a})$, or appealing to induction, $\mc{C}\models\phi(\ol{a})$. Because $\phi$ is finitary, this is true if and only if for every such $\mc{B}$, $\mc{B}\models\phi(\ol{a})$, or equivalently, if $\mc{A}\models\phi(\ol{a})$, for each $\phi\in\Phi$. This, in turn, is true if and only if $\mc{A}\models\psi(\ol{a})$.
		
		For clause (6), let $\psi = \forall\ol{y}\phi(\ol{xy})$. If $\phi$ is finitary, $\mc{A}\forces \forall\ol{y}\phi(\ol{ay})$ if and only if for every $\mc{B}\succeq\mc{A}$, and every $\ol{b}\in\mc{B}$, there is a $\mc{C}\succeq \mc{B}$ such that $\mc{C}\forces \phi(\ol{ab})$. Appealing to induction, this is true if and only if for every $\mc{B}\succeq\mc{A}$, $\ol{b}\in\mc{B}$, there is a $\mc{C}\succeq\mc{B}$ such that $\mc{C}\models\phi(\ol{ab})$. Because $\phi$ is finitary, this is true if and only if for every $\mc{B}\succeq\mc{A}$, $\ol{b}\in\mc{B}$, $\mc{B}\models \phi(\ol{ab})$, or equivalently, for every $\mc{B}\succeq\mc{A}$, $\mc{B}\models\forall\ol{y}\phi(\ol{ay})$. This, in turn, is true if and only if $\mc{A}\models\forall\ol{y}\phi(\ol{ay})$.
	\end{proof}
	
	For infinitary formulas, the relation $\forces$ is more stable than the satisfaction relation with respect to elementary extensions. (Contrast the Lemma below with Theorem \ref{theorem:alternation}.)
	
	\begin{lemma}\label{lemma:persistence}
		If $\mc{A}\preceq\mc{B}$, and $\mc{A}\forces\psi(\ol{a})$, then $\mc{B}\forces\psi(\ol{a})$.
	\end{lemma}
	\begin{proof}
		We will prove this by induction on the complexity of $\psi$. If $\psi$ is atomic, this is trivial. If $\psi = \neg\phi$, and $\mc{A}\forces\psi(\ol{a})$, then for any $\mc{C}\succeq \mc{B}$, $\mc{C}\succeq\mc{A}$, so $\mc{C}\not\forces \phi(\ol{a})$. Therefore, $\mc{B}\forces \psi(\ol{a})$. If $\psi = \bigdoublevee\limits_{\phi\in\Phi}\phi$, this follows by induction. If $\psi = \bigdoublewedge\limits_{\phi\in\Phi}\phi$, and $\mc{A}\forces\psi(\ol{a})$, then for any $\mc{C}\succeq\mc{B}$, and $\phi\in\Phi$, $\mc{C}\succeq\mc{A}$, so there is a $\mc{D}\succeq\mc{C}$ such that $\mc{D}\forces\phi(\ol{a})$. Therefore, $\mc{B}\forces\psi(\ol{a})$.
		If $\psi(\ol{x}) = \exists\ol{y}\phi(\ol{xy})$, and $\mc{A}\forces \psi(\ol{a})$, then for some $\ol{b}\in \mc{A}$, $\mc{A}\forces \phi(\ol{ab})$. Appealing to induction, $\mc{B}\forces\phi(\ol{ab})$, so $\mc{B}\forces \psi(\ol{a})$. If $\psi(\ol{x})=\forall\ol{y}\phi(\ol{xy})$, and $\mc{A}\forces \psi(\ol{a})$. Suppose $\mc{C}\succeq \mc{B}$, and $\ol{c}\in\mc{C}$. Then, $\mc{C}\succeq \mc{A}$, so because $\mc{A}\forces\psi(\ol{a})$, there is a $\mc{D}\succeq\mc{C}$ such that $\mc{D}\forces \phi(\ol{ac})$. Therefore, $\mc{B}\forces \psi(\ol{a})$.
	\end{proof}
	
	\begin{example}
		Let $\mc{A}$ and $\psi$ be the structure and formula respectively from Theorem \ref{theorem:alternation}. Then $\mc{A}\forces\psi$ (and so this is true for all elementary extensions of $\mc{A}$ as well).
	\end{example}
	\begin{proof}
		Recall that 
		\[\psi = \forall x \left(Q(x)\rightarrow \exists y\left(R(x,y)\wedge\bigdoublewedge_n \neg P_n(y)\right)\right)\]
		Suppose $\mc{B}\succeq\mc{A}$, and $a\in\mc{B}$. It suffices to show that for some $\mc{C}\succeq\mc{A}$
		\[\mc{C}\forces \neg Q(a)\vee \exists y(R(a,x)\wedge\bigdoublewedge\limits_n \neg P_n(y))\]
		In the case that $\mc{B}\models \neg Q(a)$, we can take $\mc{C}=\mc{B}$, so it suffices to consider the case that $\mc{B}\models Q(a)$. In this case, let $\mc{C}$ be obtained by adding a non-standard element $b$, in the sense of Theorem \ref{theorem:alternation} to the block corresponding to $a$. We will show that 
		\[\mc{C}\forces R(a,b)\wedge\bigdoublewedge\limits_n \neg P_n(b)\]
		It suffices to show that $\mc{C}\forces R(a,b)$, and for each $n$, $\mc{C}\forces\neg P_n(b)$. This is true because $\mc{C}\models R(a,b)$, and $\mc{C}\models \neg P_n(b)$ for each $n$.
	\end{proof} 
	
	For a structure $\mc{A}$, $\ol{a}\in \mc{A}$, and a formula $\psi$, it is immediate from the definition of $\mc{A}\forces\neg\psi(\ol{a})$ that it cannot be the case that $\mc{A}\forces\neg\psi(\ol{a})$ and $\mc{A}\forces \psi(\ol{a})$.
	We say that $\mc{A}$ decides $\psi(\ol{a})$ if either $\mc{A}\forces \psi(\ol{a})$ or $\mc{A}\forces \neg\psi(\ol{a})$.
	
	\begin{lemma}\label{lemma:decision}
		For any structure $\mc{A}$, $\ol{a}\in \mc{A}$, and formula $\psi(\ol{x})$, there is a $\mc{B}\succeq \mc{A}$ such that $\mc{B}$ decides $\psi(\ol{a})$.
	\end{lemma}
	\begin{proof}
		If $\mc{A}\forces \neg\psi(\ol{a})$, we can take $\mc{B}=\mc{A}$. Otherwise, there is some $\mc{B}\succeq \mc{A}$ with $\mc{B}\forces\psi(\ol{a})$.
	\end{proof}
	
	\subsubsection{Generic Structures}
	In order to obtain useful information from the forcing relation, we will construct structures in which formulas we have forced become true. Let $\mathbb{A}$ be a fragment of $\mc{L}_{\infty,\omega}$. We say that a structure $\mc{G}$ is $\mathbb{A}$-generic if for any $\psi\in\mathbb{A}$, and $\ol{a}\in\mc{G}$, $\mc{G}$ decides $\psi(\ol{a})$. The next lemma shows that any structure can be extended to a generic structure.
	
	\begin{lemma}
		For any structure $\mc{A}$ and fragment $\mathbb{A}$, there is an $\mathbb{A}$-generic $\mc{G}\succeq\mc{A}$.
	\end{lemma}
	
	\begin{proof}
		Let $\mc{C}$ be a structure. We will construct a structure $F(\mc{C})$ extending $\mc{C}$ as follows. Consider the set of pairs $(\psi(\ol{x}),\ol{c})$ with $\psi\in \mathbb{A}$, and $\ol{c}\in \mc{C}$ of length $\ol{x}$. Let $\{(\psi_\alpha(\ol{x}),\ol{c}_\alpha)|\alpha<\gamma\}$ be a well ordering of this set. We will define an elementary chain of length $\gamma$ by transfinite recursion. Let $\mc{C}_0 = \mc{C}$. Having defined $\mc{C}_\alpha$, we define $\mc{C}_{\alpha+1}$ as follows. By Lemma \ref{lemma:decision}, there is some $\mc{B}\succeq \mc{C}_\alpha$ that decides $\psi_\alpha(\ol{c}_\alpha)$. Let $\mc{C}_{\alpha+1} = \mc{B}$. For limit ordinals $\beta < \gamma$, let $\mc{C}_\beta = \bigcup_{\alpha<\beta}\mc{C}_\alpha$. This defines an elementary chain $\{\mc{C}_\alpha|\alpha<\gamma\}$.
		
		Let $F(\mc{C}) = \bigcup_{\alpha<\gamma}\mc{C}_\alpha$. Then, $F(\mc{C})\succeq \mc{C}$, and for every $\psi(\ol{x})\in \mathbb{A}$, $\ol{c}\in\mc{C}$, $F(\mc{C})$ decides $\psi(\ol{c})$. Now consider the elementary chain
		\[\mc{A}\preceq F(\mc{A})\preceq F^2(\mc{A})\preceq\dots\]
		Let $\mc{G} = \bigcup_n F^n(\mc{A})$. Then, $\mc{A}\preceq\mc{G}$. For any $\psi(\ol{x})\in \mathbb{A}$, $\ol{a}\in\mc{G}$, $\ol{a}\in F^n(\mc{A})$ for some $n$, so $F^{n+1}(\mc{A})$ decides $\psi(\ol{a})$, which implies that $\mc{G}$ decides $\psi(\ol{a})$ because $F^{n+1}(\mc{A})\preceq \mc{G}$. Therefore, $\mc{G}$ is $\mathbb{A}$-generic.
	\end{proof}
	
	The next lemma shows that generic structures have the desired property, providing models of formulas we have forced.
	
	\begin{lemma}\label{lemma:forcing-is-truth-in-generics}
		Suppose $\psi(\ol{x})\in\mathbb{A}$, $\mc{G}$ is $\mathbb{A}$-generic, and $\ol{a}\in\mc{G}$. Then, $\mc{G}\forces \psi(\ol{a})$ if and only if $\mc{G}\models \psi(\ol{a})$.
	\end{lemma}
	\begin{proof}
		We prove this by induction on the complexity of $\psi$. For $\psi$ atomic, this is true by definition. Suppose $\psi(\ol{x})=\neg\phi(\ol{x})$. Then $\mc{G}\forces\psi(\ol{a})$ if and only if $\mc{G}\not\forces \phi(\ol{a})$, because $\mc{G}$ is $\mathbb{A}$-generic. Appealing to induction, this is true if and only if $\mc{G}\not\models \phi(\ol{a})$, or equivalently, if $\mc{G}\models \neg\phi(\ol{a})$. For $\psi = \bigdoublevee\limits_{\phi\in\Phi}\phi$, or $\psi = \exists\ol{y}\phi(\ol{y})$, the defining clause of $\forces$ is identical to that of the satisfaction relation, and the claim follows by induction.
		
		Suppose $\psi = \bigdoublewedge\limits_{\phi\in\Phi}\phi$. If $\mc{G}\forces\psi(\ol{a})$, then for any $\phi$, there is a $\mc{B}\succeq\mc{G}$ such that $\mc{B}\forces\phi(\ol{a})$. Because $\mc{G}$ decides $\phi(\ol{a})$, it must be that $\mc{G}\forces\phi(\ol{a})$, so appealing to induction, $\mc{G}\models\phi(\ol{a})$. We conclude that $\mc{G}\models\psi(\ol{a})$. Suppose conversely that $\mc{G}\models\psi(\ol{a})$. Then, for each $\phi\in\Phi$, $\mc{G}\models\phi(\ol{a})$, so appealing to induction, $\mc{G}\forces\phi(\ol{a})$. For any $\mc{B}\succeq\mc{G}$, Lemma \ref{lemma:persistence} implies that $\mc{B}\forces\phi(\ol{a})$. We conclude that $\mc{G}\forces\psi(\ol{a})$.
		
		Suppose $\psi(\ol{x}) = \forall\ol{y}\phi(\ol{xy})$. If $\mc{G}\forces \psi(\ol{a})$, then for every $\ol{b}\in\mc{G}$, there is a $\mc{B}\succeq \mc{G}$ such that $\mc{B}\forces \phi(\ol{ab})$. In this case, $\mc{G}\not\forces\neg\phi(\ol{ab})$, so $\mc{G}\forces \phi(\ol{ab})$. Appealing to induction, $\mc{G}\models\phi(\ol{ab})$. We conclude that $\mc{G}\models \psi(\ol{a})$. Suppose conversely that $\mc{G}\not\forces\psi(\ol{a})$. Then there is some $\mc{B}\succeq\mc{G}$, and $\ol{b}\in\mc{B}$, such that for any $\mc{C}\succeq\mc{B}$, $\mc{C}\not\forces \phi(\ol{ab})$. Let $\mc{G}^\prime\succeq\mc{B}$ be $\mathbb{A}$-generic. Then, $\mc{G}^\prime\forces\neg\phi(\ol{ab})$. This implies that $\mc{G}^\prime\forces \exists\ol{y}\neg\phi(\ol{ay})$, so $\mc{G}\forces \exists\ol{y}\neg\phi(\ol{ay})$. That is, there is some $\ol{b}\in\mc{G}$ such that $\mc{G}\forces\neg\phi(\ol{ab})$. In this case, $\mc{G}\not\forces \phi(\ol{ab})$, so appealing to induction, $\mc{G}\not\models \phi(\ol{ab})$. Then, $\mc{G}\not\models\psi(\ol{a})$.
	\end{proof}
	
	It may seem as though we are making arbitrary choices about which formulas to force when we construct a generic extension of a structure $\mc{A}$. However, due to elementary amalgamation, these choices can only be made in one way. We highlight this as one of the key special features of this forcing.
	
	\begin{lemma}\label{lemma:no-choices}
		Let $\mc{A}$ be a structure, $\ol{a}\in\mc{A}$, and $\psi(\ol{x})$ be a formula. If $\mc{B}\succeq\mc{A}$, and $\mc{B}\forces \psi(\ol{a})$, then for every $\mc{C}\succeq \mc{A}$ that decides $\psi(\ol{a})$, $\mc{C}\forces\psi(\ol{a})$.
	\end{lemma}
	
	\begin{proof}
		Suppose $\mc{A}\preceq\mc{C}$ and $\mc{C}\forces\neg\psi(\ol{a})$. By the elementary amalgamation theorem, there is a $\mc{D}$ such that $\mc{B}\preceq\mc{D}$, and $\mc{C}\preceq\mc{D}$. Then, Lemma \ref{lemma:persistence} implies that $\mc{D}\forces \psi(\ol{a})$, and $\mc{D}\forces \neg\psi(\ol{a})$, which is a contradiction.
	\end{proof}
	
	Because of Lemma \ref{lemma:no-choices}, we can regard all the information as to which formulas will be forced by extensions of a structure $\mc{A}$ as already present in $\mc{A}$. The next section defines a relation that captures this information.
	
	\subsection{The Weak Forcing Relation}
	
	It will be convenient to work with the weak forcing relation, denoted $\mc{A}\forces^* \psi(\ol{a})$. This can be defined in a variety of ways, all of which are equivalent. We will provisionally define $\mc{A}\forces^* \psi$ to be $\mc{A}\forces\neg\neg\psi$. It is a standard fact that the weak forcing relation satisfies the following recursive clauses:
	\begin{enumerate}
		\item If $\psi$ is atomic, $\mc{A}\forces^* \psi(\ol{a})$ if and only if $\mc{A}\models \psi(\ol{a})$.
		\item If $\psi(\ol{x}) = \neg\phi(\ol{x})$, $\mc{A}\forces^*\psi(\ol{a})$ if and only if for every $\mc{B}\succeq\mc{A}$, $\mc{B}\not\forces^* \phi(\ol{a})$.
		\item If $\psi(\ol{x}) = \bigdoublevee\limits_{\phi\in \Phi}\phi(\ol{x})$, $\mc{A}\forces^*\psi(\ol{a})$ if and only if for each $\mc{B} \succeq \mc{A}$, there is $\mc{C} \succeq \mc{B}$ and  $\phi\in\Phi$ such that $\mc{C}\forces^*\phi(\ol{a})$.
		\item If $\psi(\ol{x}) = \bigdoublewedge\limits_{\phi \in\Phi}\phi(\ol{x})$, $\mc{A}\forces^*\psi(\ol{a})$ if and only if for every $\phi\in \Phi$, $\mc{A}\forces^*\phi(\ol{a})$.
		\item If $\psi(\ol{x}) = \exists\ol{y}\phi(\ol{xy})$, $\mc{A}\forces^*\psi(\ol{a})$ if and only if for all $\mc{B}\succeq \mc{A}$ there is $\mc{C} \succeq \mc{B}$ and $\ol{b}\in \mc{C}$ such that $\mc{C}\forces^* \phi(\ol{ab})$.
		\item If $\psi(\ol{x}) = \forall\ol{y}\phi(\ol{xy})$, $\mc{A}\forces^*\psi(\ol{a})$ if and only if for every $\mc{B} \succeq \mc{A}$ and $\ol{b}\in \mc{B}$, $\mc{B}\forces^* \phi(\ol{ab})$.
	\end{enumerate}
	
	The following lemma establishes other equivalent characterizations of the weak forcing relation, which are unique to this notion of forcing and result from elementary amalgamation. (In essence, elementary amalgamation allows us to simplify being ``dense below''.)
	
	\begin{lemma}\label{lemma:equivalent-characterizations}
		Let $\psi$ be an $\mc{L}_{\infty,\omega}$ formula. The following are equivalent.
		\begin{enumerate}
			\item $\mc{A}\forces^*\psi(\ol{a})$.
			\item For some $\mc{B}\succeq\mc{A}$, $\mc{B}\forces\psi(\ol{a})$.
			\item For every $\mc{B}\succeq\mc{A}$ that decides $\psi(\ol{a})$, $\mc{B}\forces\psi(\ol{a})$
			\item If $\psi\in\mathbb{A}$, and $\mc{G}\succeq \mc{A}$ is $\mathbb{A}$-generic, $\mc{G}\models\psi(\ol{a})$.
		\end{enumerate}
	\end{lemma}
	
	\begin{proof}
		First, we will show that (1) implies (2).
		If $\mc{A}\forces\neg\neg\psi(\ol{a})$, then $\mc{A}\not\forces\neg\psi(\ol{a})$, so there is a $\mc{B}\succeq\mc{A}$ such that $\mc{B}\forces\psi(\ol{a})$.
		That (2) implies (3) follows from Lemma \ref{lemma:no-choices}.
		
		Now, we will show that (3) implies (4).
		If $\psi\in\mathbb{A}$, and $\mc{G}\succeq\mc{A}$ is $\mathbb{A}$-generic, then $\mc{G}\succeq\mc{A}$ decides $\psi(\ol{a})$. If (3) holds, $\mc{G}\forces\psi(\ol{a})$, so by Lemma \ref{lemma:forcing-is-truth-in-generics}, $\mc{G}\models\psi(\ol{a})$. 
		
		Next, we will show that (4) implies (1).
		Suppose that if $\psi\in\mathbb{A}$ and $\mc{G}\succeq\mc{A}$ is $\mathbb{A}$-generic, $\mc{G}\models\psi(\ol{a})$. By Lemma \ref{lemma:forcing-is-truth-in-generics}, $\mc{G}\forces\psi(\ol{a})$. Suppose $\mc{B}\succeq\mc{A}$. By Lemma \ref{lemma:decision}, there is some $\mc{C}\succeq\mc{B}$ that decides $\psi(\ol{a})$. By Lemma \ref{lemma:no-choices}, $\mc{C}\forces\psi(\ol{a})$. Consequently, for all $\mc{B}\succeq\mc{A}$, $\mc{B}\not\forces\neg\psi(\ol{a})$, so $\mc{A}\forces\neg\neg\psi(\ol{a})$.
	\end{proof}
	
	From the second characterization of the weak forcing relation, we have the following.
	
	\begin{corollary}\label{corollary:strong-implies-weak}
		If $\mc{A}\forces \psi(\ol{a})$, then $\mc{A}\forces^*\psi(\ol{a})$.
	\end{corollary}
	
	We can now establish some useful properties enjoyed by the weak forcing relation.
	
	\begin{lemma}\label{lemma:transfer}
		If $\mc{A}\preceq\mc{B}$, then for any $\psi\in\mc{L}_{\infty,\omega}$, and $\ol{a}\in\mc{A}$, $\mc{A}\forces^*\psi(\ol{a})$ if and only if $\mc{B}\forces^*\psi(\ol{a})$.
	\end{lemma}
	
	\begin{proof}
		Suppose that $\mc{A}\forces^*\psi(\ol{a})$. Then, if $\mc{C}\succeq\mc{B}$, and $\mc{C}$ decides $\psi(\ol{a})$, $\mc{C}\succeq\mc{A}$, so $\mc{C}\forces\psi(\ol{a})$. We conclude that $\mc{B}\forces^* \psi(\ol{a})$. Suppose now that $\mc{B}\forces^*\psi(\ol{a})$. Then, for some $\mc{C}\succeq\mc{B}$, $\mc{C}\forces \psi(\ol{a})$. $\mc{C}\succeq\mc{A}$, so $\mc{A}\forces^*\psi(\ol{a})$.
	\end{proof}
	
	As a consequence of elementary amalgamation, we get the following lemma which is an important and special property of this forcing notion.
	
	\begin{lemma}\label{lemma:completeness}
		For any $\mc{A}$, $\ol{a}\in\mc{A}$, and formula $\psi$, either $\mc{A}\forces^*\psi(\ol{a})$ or $\mc{A}\forces^*\neg\psi(\ol{a})$. 
	\end{lemma}
	
	\begin{proof}
		By Lemma \ref{lemma:decision}, there is some $\mc{B}\succeq\mc{A}$ such that either $\mc{B}\forces\psi(\ol{a})$, or $\mc{B}\forces\neg\psi(\ol{a})$. In the first case, $\mc{A}\forces^*\psi(\ol{a})$, and in the second, $\mc{A}\forces^*\neg\psi(\ol{a})$
	\end{proof}
	
	This implies that $\mc{A}\forces^*\neg\phi(\ol{a})$ if and only if $\mc{A}\not\forces^*\phi(\ol{a})$. Moreover, using Lemma \ref{lemma:transfer} and elementary amalgamation:
	
	\begin{theorem}\label{rem:weak-forcing-new-defn}
		The weak forcing $\forces^*$ is defined by the following recursive conditions:
		\begin{enumerate}
			\item If $\psi$ is atomic, $\mc{A}\forces^* \psi(\ol{a})$ if and only if $\mc{A}\models \psi(\ol{a})$.
			\item[($2'$)] If $\psi(\ol{x}) = \neg\phi(\ol{x})$, $\mc{A}\forces^* \psi(\ol{a})$ if and only $\mc{A}\not\forces^* \phi(\ol{a})$.
			\item[($3'$)] If $\psi(\ol{x}) = \bigdoublevee\limits_{\phi\in \Phi}\phi(\ol{x})$, $\mc{A}\forces^*\psi(\ol{a})$ if and only if for some $\phi\in\Phi$, $\mc{A}\forces^*\phi(\ol{a})$.
			\item[($4$)] If $\psi(\ol{x}) = \bigdoublewedge\limits_{\phi \in\Phi}\phi(\ol{x})$, $\mc{A}\forces^*\psi(\ol{a})$ if and only if for every $\phi\in \Phi$, $\mc{A}\forces^*\phi(\ol{a})$.
			\item[($5'$)] If $\psi(\ol{x}) = \exists\ol{y}\phi(\ol{xy})$, $\mc{A}\forces^*\psi(\ol{a})$ if and only if there is $\mc{B} \succeq \mc{A}$ and $\ol{b}\in \mc{B}$ such that $\mc{B}\forces^*\phi(\ol{ab})$.
			\item[($6$)] If $\psi(\ol{x}) = \forall\ol{y}\phi(\ol{xy})$, $\mc{A}\forces^*\psi(\ol{a})$ if and only if for every $\mc{B} \succeq \mc{A}$ and $\ol{b}\in \mc{B}$, $\mc{B}\forces^* \phi(\ol{ab})$.
		\end{enumerate}
	\end{theorem}
	
	\begin{proof}
		As noted previously, $\forces^*$ satisfies clauses (1), (4) and (6). That $\forces^*$ satisfies ($2'$) follows immediately from Lemma \ref{lemma:completeness}.
		
		For ($3'$), suppose $\psi = \bigdoublevee\limits_{\phi\in\Phi}\phi$, and that $\mc{A}\forces^*\psi(\ol{a})$. Then, there is a $\mc{B}\succeq\mc{A}$ such that $\mc{B}\forces \psi(\ol{a})$. Consequently, $\mc{B}\forces \phi(\ol{a})$ for some $\phi\in\Phi$, so $\mc{A}\forces^*\phi(\ol{a})$. Suppose conversely that $A\forces^* \phi(\ol{a})$ for some $\phi\in\Phi$. Then, for some $\mc{B}\succeq\mc{A}$, $\mc{B}\forces\phi(\ol{a})$, so $\mc{B}\forces\psi(\ol{a})$. We conclude that $\mc{A}\forces^*\psi(\ol{a})$.
		
		For ($5'$), suppose $\psi(\ol{x}) = \exists\ol{y}\phi(\ol{xy})$, and that $\mc{A}\forces^*\psi(\ol{a})$. Then, there is a $\mc{B}\succeq\mc{A}$ such that $\mc{B}\forces\psi(\ol{a})$, so there is a $\ol{b}\in\mc{B}$ such that $\mc{B}\forces\phi(\ol{ab})$, in which case $\mc{B}\forces^*\phi(\ol{ab})$. Conversely, if for some $\mc{B}\succeq\mc{A}$, $\ol{b}\in\mc{B}$, $\mc{B}\forces^*\phi(\ol{ab})$, then for some $\mc{C}\succeq\mc{B}$, $\mc{C}\forces \phi(\ol{ab})$. In this case, $\mc{C}\forces\psi(\ol{a})$ so $\mc{A}\forces^*\psi(\ol{a})$.
	\end{proof}
	
	Note that ($1$), ($2'$), ($3'$), and ($4$) are the same as the conditions for the satisfaction relation; it is only ($5'$) and ($6$) that differ. Only the quantifiers, and not the infinitary connectives, are treated differently; this is why we use the $\exists_n / \forall_n$ hierarchy rather than the $\Sigma_n/\Pi_n$ hierarchy. Moreover, we obtain the following.
	
	\begin{corollary}\label{corollary:quantifier-free-weak-forcing-is-truth}
		If $\psi$ is quantifier-free, $\mc{A}\forces^*\psi(\ol{a})$ if and only if $\mc{A}\models\psi(\ol{a})$.
	\end{corollary}
	
	The following lemma shows that the weak forcing relation respects entailment and equivalence of formulas.
	
	\begin{lemma}\label{lemma:monotonicity}
		If $\psi_1\models \psi_2$ and $\mc{A}\forces^*\psi_1(\ol{a})$, then $\mc{A}\forces^*\psi_2(\ol{a})$.
	\end{lemma}
	
	\begin{proof}
		Suppose that $\mc{A}\forces^*\psi_1(\ol{a})$ and $\mc{A}\not\forces^*\psi_2(\ol{a})$. By Lemma \ref{lemma:completeness}, $\mc{A}\forces^*\neg\psi_2(\ol{a})$. Let $\mathbb{A}$ be a fragment containing $\psi_1$ and $\neg\psi_2$, and let $\mc{G}\succeq\mc{A}$ be $\mathbb{A}$-generic. Then, $\mc{G}\models \psi_1(\ol{a})$ and $\mc{G}\models\neg\psi_2(\ol{a})$, so $\psi_1\nmodels\psi_2$.
	\end{proof}
	
	We can also define generic structures in terms of the weak forcing relation.
	
	\begin{lemma}
		$\mc{G}$ is $\mathbb{A}$-generic if and only if for $\psi(\ol{x})=\exists\ol{y}\phi(\ol{xy})\in\mathbb{A}$, $\ol{a}\in \mc{G}$, if $\mc{G}\forces^* \psi(\ol{a})$, there is a $\ol{b}\in\mc{G}$ such that $\mc{G}\forces^* \phi(\ol{ab})$.
	\end{lemma}
	
	\begin{proof}
		Suppose $\mc{G}$ is $\mathbb{A}$-generic. Let $\psi(\ol{x})=\exists\ol{y}\phi(\ol{xy})\in \mathbb{A}$, and $\ol{a}\in\mc{G}$, If $\mc{G}\forces^*\psi(\ol{a})$, then $\mc{G}\models\psi(\ol{a})$, so for some $\ol{b}\in\mc{G}$, $\mc{G}\models\phi(\ol{ab})$. By Lemma \ref{lemma:forcing-is-truth-in-generics}, $\mc{G}\forces\phi(\ol{ab})$, so $\mc{G}\forces^*\phi(\ol{ab})$. 
		
		Conversely, suppose that for any $\psi(\ol{x}) = \exists\ol{y}\phi(\ol{xy})\in\mathbb{A}$, $\ol{a}\in\mc{G}$, if $\mc{G}\forces^*\psi(\ol{a})$, then for some $\ol{b}\in\mc{G}$, $\mc{G}\forces^*\phi(\ol{ab})$. We will show that $\mc{G}$ is $\mathbb{A}$-generic. By Lemma \ref{lemma:completeness}, it suffices to show that if $\psi\in\mathbb{A}$, and $\mc{G}\forces^* \psi(\ol{a})$, then $\mc{G}\forces \psi(\ol{a})$. We will prove this by induction on the complexity of $\psi$.
		
		If $\psi$ is atomic, then $\mc{A}\forces^* \psi(\ol{a})$ if and only $\mc{A}\forces\psi(\ol{a})$ if and only if $\mc{A}\models \psi(\ol{a})$. If $\psi = \neg \phi$, and $\mc{G}\forces^*\psi(\ol{a})$, then for every $\mc{B}\succeq\mc{G}$, $\mc{B}\not\forces^*\phi(\ol{a})$. By Corollary \ref{corollary:strong-implies-weak}, for every such $\mc{B}$, $\mc{B}\not\forces \phi(\ol{a})$, so $\mc{G}\forces \psi(\ol{a})$. If $\psi = \bigdoublevee\limits_{\phi\in\Phi}\phi$, and $\mc{G}\forces^* \psi(\ol{a})$, then $\mc{G}\forces^*\phi(\ol{a})$ for some $\phi\in\Phi$. Appealing to induction, $\mc{G}\forces\phi(\ol{a})$, so $\mc{G}\forces\psi(\ol{a})$. Likewise, if $\psi = \bigdoublewedge\limits_{\phi\in\Phi}\phi$ and $\mc{G}\forces^*\psi(\ol{a})$, then $\mc{G}\forces^*\phi(\ol{a})$ for every $\phi\in\Phi$, so $\mc{G}\forces\phi(\ol{a})$ for every $\phi\in \Phi$, which implies that $\mc{G}\forces \psi(\ol{a})$.
		
		Suppose that $\psi(\ol{x}) = \exists\ol{y}\phi(\ol{xy})$. If $\mc{G}\forces^*\psi(\ol{a})$, then for some $\ol{b}\in\mc{G}$, $\mc{G}\forces^*\phi(\ol{ab})$. Appealing to induction, $\mc{G}\forces \phi(\ol{ab})$, so $\mc{G}\forces \psi(\ol{ab})$. Suppose $\psi(\ol{x}) = \forall\ol{y}\phi(\ol{xy})$, and $\mc{G}\forces^*\psi(\ol{a})$. Then, for any $\mc{B}\succeq\mc{G}$, $\ol{b}\in\mc{B}$, $\mc{B}\forces^*\phi(\ol{ab})$. Let $\mc{C}\succeq\mc{B}$ decide $\phi(\ol{ab})$. Then, $\mc{C}\forces\phi(\ol{ab})$. We conclude that $\mc{G}\forces\psi(\ol{a})$.
	\end{proof}
	
	The following lemma shows that the weak forcing relation depends only on finitary first order properties, namely whether $\mc{A}$ forces that $\phi$ is true of $\ol{a}$ depends only on the type of $\ol{a}$ in $\mc{A}$.
	
	\begin{lemma}\label{Invariance}
		Suppose $(\mc{A},\ol{a})\equiv(\mc{B},\ol{b})$, then, for any formula $\psi$, $\mc{A}\forces^*\psi(\ol{a})$ if and only if $\mc{B}\forces^*\psi(\ol{b})$.
	\end{lemma}
	
	\begin{proof}
		If $(\mc{A},\ol{a})\equiv (\mc{B},\ol{b})$, the elementary amalgamation theorem implies that there is a structure $\mc{C}$ and elementary embeddings $f:\mc{A}\hookrightarrow\mc{C}$ and $g:\mc{B}\hookrightarrow\mc{C}$ such that $f(\ol{a}) = g(\ol{b})$. We can then identify $\mc{A}$ and $\mc{B}$ with elementary substructures of $\mc{C}$ so that $\ol{a} = \ol{b}$. Using Lemma \ref{lemma:transfer}, we have that $\mc{A}\forces^* \psi(\ol{a})$ if and only if $\mc{C}\forces^*\psi(\ol{a})$ if and only if $\mc{B}\forces^*\psi(\ol{b})$.
	\end{proof}
	
	\subsubsection{Definability}
	
	 Lemma \ref{Invariance} showed that the weak forcing relation depends only on finitary first order properties. The next lemma shows that, moreover, quantifier complexity is maintained. Recall that we said that a formula $\psi\in\mc{L}_{\infty,\omega}$ is elementary if it is of the form $\psi=\bigdoublevee_{\alpha}\bigdoublewedge_{\beta}\theta_{\alpha,\beta}$, for $\theta_{\alpha,\beta}$ finitary formulas.
	
	\begin{lemma}
		For each $\psi\in\mc{L}_{\infty,\omega}$, there is an elementary formula $\mathrm{Force}_\psi$ such that $\mc{A}\forces^*\psi(\ol{a})$ if and only if $\mc{A}\models \mathrm{Force}_\psi(\ol{a})$. Moreover, if $\psi$ is a $\forall_n$ (resp.\ $\exists_n$) formula, then $\mathrm{Force}_\psi(\ol{a})$ can be taken to be of the form
		\[ \bigdoublevee_\alpha \bigdoublewedge_\beta \theta_{\alpha,\beta}(\ol{x}) \]
		where each $\theta_{\alpha,\beta}$ is a finitary $\forall_n$ (resp. $\exists_n$) formula.
	\end{lemma}
	
	Without the last clause, the lemma follows quite simply from Lemma \ref{Invariance}. Consider the following set of types $\mc{T}$. Let $\mc{T} = \{ \tp^{\mc{A}}(\ol{a}) : \mc{A} \forces^* \psi(\ol{a})\}$ where $\mc{A}$ ranges over all structures. By Lemma \ref{Invariance}, $\mc{B} \forces^* \psi(\ol{a})$ if and only if, for some $p \in \mc{T}$, $\mc{B} \models p(\ol{a})$. Then let
	\[\mathrm{Force}_\psi(\ol{x}) = \bigdoublevee_{p(\ol{x}) \in \mc{T}} \bigdoublewedge_{\varphi \in p(\ol{x})} \varphi(\ol{x}).\]
	However, we need a more involved argument if we want $\mathrm{Force}_\psi$ to have the same quantifier complexity as $\psi$.
	
	\begin{proof}
		We will define $\mathrm{Force}_\psi$ by recursion (recalling Theorem \ref{rem:weak-forcing-new-defn} which gives simplified conditions for the weak forcing). At each step, we will ensure that $\mathrm{Force}_\psi$ is at most the complexity of $\psi$. If $\psi$ is atomic, let $\mathrm{Force}_\psi = \psi$. Suppose $\psi = \neg\phi$. By Lemma $\ref{lemma:completeness}$, $\mc{A}\forces^* \psi(\ol{a})$ if and only if $\mc{A}\not\forces^*\phi(\ol{a})$. Let $\mathrm{Force}_\phi = \bigdoublevee_\alpha \bigdoublewedge_\beta \theta_{\alpha,\beta}$. Then, $\mc{A}\forces^* \psi(\ol{a})$ if and only if $\mc{A}\models \sim \mathrm{Force}_\phi(\ol{a})$, the formal negation of $\mathrm{Force}_\phi$. 
		This is
		\[\sim \mathrm{Force}_\phi = \bigdoublewedge_\alpha\bigdoublevee_\beta \sim \theta_{\alpha,\beta}\]
		which is equivalent to
		\[\bigdoublevee_{f:\alpha\mapsto\beta}\bigdoublewedge_\alpha \sim\theta_{\alpha,f(\alpha)}\]
		We define $\mathrm{Force}_\psi$ to be this.
		
		Suppose $\psi = \bigdoublevee\limits_{\phi\in\Phi}\phi$. We can then define $\mathrm{Force}_\psi$ as $\bigdoublevee\limits_{\phi\in\Phi}\mathrm{Force}_\phi$. Suppose $\psi = \bigdoublewedge\limits_{\phi\in\Phi}\phi$. Then, $\mc{A}\forces^*\psi(\ol{a})$ if and only if for every $\phi\in\Phi$, $\mc{A}\models\mathrm{Force}_{\phi}(\ol{a})$. Let $\mathrm{Force}_\phi = \bigdoublevee_\alpha\bigdoublewedge_\beta \theta^\phi_{\alpha,\beta}$. Then, $\mc{A}\forces^*\psi(\ol{a})$ if and only if
		\[\mc{A}\models \bigdoublewedge_{\phi\in\Phi}\bigdoublevee_\alpha\bigdoublewedge_{\beta}\theta^\phi_{\alpha,\beta}(\ol{a})\]
		This formula is equivalent to
		\[\bigdoublevee_{f:\phi\mapsto\alpha}\bigdoublewedge_{\phi\in\Phi,\beta}\theta^{\phi}_{f(\phi),\beta}\]
		We define $\mathrm{Force}_\psi$ to be this.
		
		Suppose $\psi(\ol{x}) = \exists\ol{y}\phi(\ol{xy})$. Let $\mathrm{Force}_{\phi} = \bigdoublevee_\alpha\bigdoublewedge_\beta \theta_{\alpha,\beta}$. The following are equivalent.
		\begin{enumerate}
			\item $\mc{A}\forces^* \psi(\ol{a})$;
			\item For some $\mc{B}\succeq\mc{A}$ and $\ol{b}\in\mc{B}$, $\mc{B}\forces^*\phi(\ol{ab})$;
			\item For some $\mc{B}\succeq\mc{A}$ and $\ol{b}\in\mc{B}$, $\mc{B}\models \bigdoublevee_{\alpha}\bigdoublewedge_\beta \theta_{\alpha,\beta}(\ol{ab})$;
			\item For some $\mc{B}\succeq\mc{A}$, $\ol{b}\in\mc{B}$, and $\alpha$, $\mc{B}\models \theta_{\alpha,\beta}(\ol{ab})$ for each $\beta$;
			\item For some $\alpha$, the partial type $p_\alpha(\ol{y}) = \left\{\theta_{\alpha,\beta}(\ol{ay})\middle| \beta\right\}$ is finitely satisfiable in $\mc{A}$;
			\item \[\mc{A}\models \bigdoublevee\limits_\alpha \bigdoublewedge\limits_{S\textrm{  finite}} \exists\ol{y}\bigwedge\limits_{\beta\in S}\theta_{\alpha,\beta}(\ol{ay}).\]
		\end{enumerate}
		We define $\mathrm{Force}_\psi$ to be this formula. Suppose $\psi(\ol{x}) = \forall\ol{y}\phi(\ol{xy})$. By Lemma \ref{lemma:monotonicity}, $\mc{A}\forces^*\psi(\ol{a})$ if and only if $\mc{A}\forces^* \neg\exists\ol{y}\neg\phi(\ol{ay})$, so we can use the rules for existential quantifiers and negations to construct $\mathrm{Force}_\psi = \mathrm{Force}_{\neg\exists\ol{y}\neg\phi(\ol{y})}$.
	\end{proof}
	
	A drawback of this definition is that for a cardinal $\kappa$, if $\psi\in\mc{L}_{\kappa,\omega}$, $\mathrm{Force}_\psi$ may not be in $\mc{L}_{\kappa,\omega}$. For instance, if $\tau$ is countable, and $\psi\in\mc{L}_{\omega_1,\omega}$, $\mathrm{Force}_\psi$ may involve a disjunction over uncountably many formulas. The next results show that this cannot be avoided.
	
	\begin{lemma}
		There is a countable signature $\tau$, and a sentence $\psi\in\mc{L}_{\omega_1,\omega}(\tau)$ such that for any tree $T\subset \omega^{<\omega}$, there is a countable $\tau$-structure $\mc{A}_T$, uniformly computable in $T$, satisfying $\mc{A}_T\forces^*\psi$ if and only if $T$ has an infinite path.
	\end{lemma}
	
	\begin{proof}
		Let $\tau$ consist of unary relation symbols $R_{i,j}$ for $i,j\in\mathbb{N}$. Let $\psi = \exists x \bigdoublewedge_i \bigdoublevee_j R_{i,j}(x)$. For a tree $T\subset\omega^{<\omega}$, we define $\mc{A}_T$ as follows. For each $\sigma\in T$, there is an element of $\mc{A}_T$ satisfying exactly the relations $R_{i,\sigma(i)}$ for each $i$ less than the length of $\sigma$.
		
		If $\mc{A}_T\forces^* \psi$, then for some $\mc{B}\succeq\mc{A}_T$, and $b\in\mc{B}$, $\mc{B}\forces^*\bigdoublewedge\limits_i\bigdoublevee\limits_j R_{i,j}(b)$. In this case, $\mc{B}\models\bigdoublewedge\limits_i\bigdoublevee\limits_j R_{i,j}(b)$.  Then, for some function $f\in\omega^\omega$, $\mc{B}\models R_{i,f(i)}(b)$ for each $i$. This implies that the partial type $\left\{R_{i,f(i)}\middle|i<\omega\right\}$ is finitely satisfiable in $\mc{A}_T$, so for every $n$, there is a $a\in\mc{A}_T$ such that $\mc{A}_T\models R_{i,f(i)}(a)$ for $i<n$. That is, $f\res n \in T$, for all $n$, so $f$ is a path in $T$. Suppose conversely that $f$ is a path in $T$. Then, the partial type $\left\{R_{i,f(i)}\middle|i<\omega\right\}$ is finitely satisfiable in $\mc{A}_T$, so for some elementary extension $\mc{B}\succeq\mc{A}_T$, there is a $b\in\mc{B}$ realizing this type. Then, $\mc{B}\models \bigdoublewedge\limits_i R_{i,f(i)}(B)$, so $\mc{B}\models \bigdoublewedge\limits_i\bigdoublevee\limits_j R_{i,j}(b)$. In this case, $\mc{B}\forces^* \bigdoublewedge\limits_i\bigdoublevee\limits_j R_{i,j}(b)$, so $\mc{A}_T\forces^* \psi$.
	\end{proof}
	
	Let $\Mod_\mc{\tau}$ be the Polish space of $\omega$-presentations of $\tau$-structures. The mapping $T\mapsto \mc{A}_T$ witnesses the following. 
	
	\begin{corollary}\label{cor:not-countable}
		The set $\left\{\mc{A}\in \Mod_{\mc{\tau}}\middle| \mc{A}\forces^*\psi\right\}$ is $\bfSigma^1_1$ hard.
	\end{corollary}
	
	We conclude that this set is not Borel, so is not the set of models of a $\mc{L}_{\omega_1,\omega}$ sentence. As such, we cannot have $\mathrm{Force}_\psi \in\mc{L}_{\omega_1,\omega}$.
	
	\subsection{Structures of Bounded Cardinality}\label{subsection:bounded-cardinality}
	The apparatus built up in the previous sections can be adapted to consider only structures of cardinality below a particular bound $\kappa$. In the recursive definitions of the strong and weak forcing relations, one replaces elementary extensions in general with those of cardinality below $\kappa$. In order to construct generic structures of cardinality below $\kappa$, one also needs that the fragment $\mathbb{A}$ satisfies $|\mathbb{A}|<\kappa$, and so consists of $\mc{L}_{\kappa,\omega}$ formulas. Otherwise, the proofs go through without any changes.
	
	\section{Applications of the Forcing Notion}
	
	We now apply the forcing with elementary extension introduced in the previous section to prove the main theorems of this paper.
	
	\subsection{The Main Theorem}
	
	In this section, we will prove Theorem \ref{theorem:interpolation}.
	
	\interpolationtheorem*
	
	To prove this, we will use the fact that a finitary formula $\varphi$ is equivalent to a $\exists_{n}$ formula over a theory $T$ if and only if whenever $\mc{A}\preceq_{n-1}\mc{B}$ are models of $T$, $\ol{a}\in\mc{A}$, and $\mc{A}\models\varphi(\ol{a})$, then $\mc{B}\models\varphi(\ol{a})$. This generalises the fact that a finitary formula is equivalent to an existential formula if and only if it is preserved upwards under superstructures.
	
	However, while an infinitary $\exists_1$ formula is preserved upwards under superstructures, it is not generally true that an infinitary $\exists_{n}$ formula is preserved upwards under $(n-1)$-elementary superstructures. Instead, we will show that they are preserved upwards under $(n-1)$-elementary superstructures if we consider weak forcing rather than satisfaction.
	
	\begin{lemma}\label{lemma:n-extensions-forcing}
		Suppose $\mc{A}\preceq_{n-1}\mc{B}$ and $\ol{a}\in\mc{A}$. Let $\psi$ be an infinitary $\exists_{n}$ formula. Then if $\mc{A}\forces^*\psi(\ol{a})$ then $\mc{B}\forces^*\psi(\ol{a})$.
	\end{lemma}
	
	\begin{proof}
		Let $\mathrm{Force}_\psi = \bigdoublevee_\alpha \bigdoublewedge_\beta \theta_{\alpha,\beta}$, where each $\theta_{\alpha,\beta}$ is a finitary $\exists_{n}$ formula. Suppose that $\mc{A}\forces^* \psi(\ol{a})$. Then $\mc{A}\models\mathrm{Force}_\psi(\ol{a})$, and so for some $\alpha$, and every $\beta$, $\mc{A}\models\theta_{\alpha,\beta}(\ol{a})$. Because each $\theta_{\alpha,\beta}$ is a finitary $\exists_{n}$ formula, $\mc{B}\models\theta_{\alpha,\beta}(\ol{a})$ for the same $\alpha$, and every $\beta$. Therefore, $\mc{B}\models\mathrm{Force}_\psi(\ol{a})$, so $\mc{B}\forces^*\psi(\ol{a})$.
	\end{proof}
	
	With Lemma \ref{lemma:n-extensions-forcing}, we can now prove Theorem \ref{theorem:interpolation}.
	
	\begin{proof}[Proof of Theorem \ref{theorem:interpolation}]
		We prove the $\exists_n$ case; the $\forall_n$ case can be obtained by taking negations.
		
		Suppose that, as in the hypotheses of Theorem \ref{theorem:interpolation}, $\psi$ is an infinitary $\exists_n$ formula which is equivalent to a finitary formula $\varphi$ in all models of $T$. We want to show that $\psi$ and $\varphi$ are equivalent to a finitary $\exists_n$ formula in all models of $T$. To do this, suppose that $\mc{A} \preceq_{n-1} \mc{B}$ are models of $T$, $\ol{a} \in A$, and $\mc{A} \models \varphi(\ol{a})$; we must show that $\mc{B} \models \varphi(\ol{a})$.
		
		Now since $\mc{A} \models \varphi(\ol{a})$ and $\varphi$ is finitary, by Lemma \ref{lemma:finitary-forcing-is-truth} we have that $\mc{A} \forces^* \varphi(\ol{a})$. Since $\varphi$ and $\psi$ are equivalent in $\mc{A}$, and forcing respects this (Lemma \ref{lemma:monotonicity}), we have that $\mc{A} \forces^* \psi(\ol{a})$. But we just proved in Lemma \ref{lemma:n-extensions-forcing} that forcing an $\exists_n$ formula is preserved upwards under $n-1$-elementary superstructures, and so $\mc{B} \forces^* \psi(\ol{a})$. Using the same equivalences as before, we get that $\mc{B} \forces^* \varphi(\ol{a})$ and then that $\mc{B} \models \varphi(\ol{a})$. This completes the argument.
	\end{proof}
	
	\begin{remark}
		Suppose that $\tau$ is a countable and $\varphi$ is a sentence of $\mc{L}_{\omega,\omega}$. 
		By the L\"owenheim-Skolem theorem for $\mc{L}_{\omega_1,\omega}$, if $\psi$ is a sentence of $\mc{L}_{\omega_1,\omega}$ and $\varphi$ and $\psi$ are equivalent in all countable structures, they are equivalent in all structures. Using Theorem \ref{theorem:interpolation} and Vaught's version of the Lopez-Escobar theorem \cite{Vaught75}, we have that the following are equivalent.
		\begin{enumerate}
			\item $\varphi$ is equivalent to a finitary $\exists_n$ sentence (respectively $\forall_n$).
			\item $\left\{\mc{A}\in\Mod_\tau\middle|\mc{A}\models \varphi\right\}$ is $\bfSigma^0_n$ (respectively, $\bfPi^0_n$).
		\end{enumerate}
		Thus Vaught's version of the Lopez-Escobar theorem specialises to the case of finitary formulas.
	\end{remark}
	
	The proof of Theorem \ref{theorem:interpolation} we gave above makes use of standard ideas that show up in forcing, like the definability of forcing. One can also give a more hands-on proof which has a different sort of explanatory power. The outline of this proof is as follows. We prove the contrapositive: Supposing that $\varphi\in\mc{L}_{\omega,\omega}$ is not equivalent to any finitary $\exists_n$ formula over $T$, we aim to produce a model witnessing that $\varphi$ is not equivalent over $T$ to some particular $\exists_n$ formula $\psi$. Using the fact that $\varphi$ is not equivalent to any finitary $\exists_n$ formula, we can construct models of $T$, $\mc{A}\preceq_{n-1} \mc{B}$, such that for some $\ol{a}\in\mc{A}$, $\mc{A}\models \varphi(\ol{a})$ and $\mc{B}\models\neg\varphi(\ol{a})$. It suffices then to construct either an elementary extension of $\mc{A}$ modeling $\neg\psi(\ol{a})$ or an elementary extension of $\mc{B}$ modeling $\psi(\ol{a})$. Considering $\mathbb{A}$-generic elementary extensions of $\mc{A}$ and $\mc{B}$ for a fragment $\mathbb{A}$ containing $\psi$, it suffices to show that either $\mc{A}\forces^*\neg\psi(\ol{a})$ or that $\mc{B}\forces^*\psi(\ol{a})$. This is the content of Lemma $\ref{lemma:n-extensions-forcing}$. This lemma can be proved by a more semantic route, using the following amalgamation lemmas, both of which are applications of compactness.
	
	\begin{lemma}\label{lemma:amalgamation-1}
		Suppose $\mc{A}\preceq_n \mc{B}$ and $\mc{A}\preceq \mc{A}^\prime$. There is a $\mc{B}^\prime$ such that $\mc{B}^\prime \succeq \mc{B}$ and $\mc{B}^\prime \succeq_n \mc{A}^\prime$.
	\end{lemma}
	
	\begin{lemma}\label{lemma:amalgamation-2}
		If $\mc{A}\preceq_{n+1} \mc{B}$, then there is a $\mc{C}\succeq_{n} \mc{B}$ with $\mc{C}\succeq \mc{A}$.
	\end{lemma}
	
	We then proceed by induction on $n$. Appealing to induction on the complexity of $\psi$, we can reduce to the case that $\psi(\ol{x})$ is of the form $\exists \ol{y}\eta(\ol{xy})$, where $\eta$ is $\forall_{n-1}$. Suppose $\mc{A}\forces^* \psi(\ol{a})$. Then, for some $\mc{A}^\prime\succeq\mc{A}$, and $\ol{a^\prime}\in\mc{A}^\prime$, $\mc{A}^\prime\forces \eta(\ol{aa^\prime})$. Applying Lemma \ref{lemma:amalgamation-1}, we have a $\mc{B}^\prime\succeq \mc{B}$, such that $\mc{A}^\prime\preceq_{n-1} \mc{B}^\prime$. Applying Lemma \ref{lemma:amalgamation-2}, there is a $\mc{C}\succeq_{n-2}\mc{B}^\prime$ such that $\mc{C}\succeq \mc{A}^\prime$. Then, $\mc{C}\forces^*\eta(\ol{aa^\prime})$. Appealing to induction on $n$, in the case of the pair of structures $\mc{B}^\prime\preceq_{n-2} \mc{C}$ and the $\exists_{n-1}$ formula $\neg\eta$, we have that $\mc{B}^\prime\forces^*\eta(\ol{aa^\prime})$, so $\mc{B}\forces^*\psi(\ol{a})$. Unraveling this argument, the recursive definition of the weak forcing relation guides a construction of a sequence of elementary extensions containing witnesses for subformulas of either $\psi$ or $\neg\psi$.
	
	\subsection{Preservation by Elementary Extensions}
	
	\subsubsection{\texorpdfstring{Preservation of formulas in $\mc{L}_{\infty,\omega}$}{Preservation of infinitary formulas}}
	
	We will now prove Theorem \ref{theorem:preservation}.
	\preservationtheorem*
	
	\begin{proof}
		Suppose that $\psi$ is equivalent to an elementary formula $\phi = \bigdoublevee_\alpha \bigdoublewedge_\beta \theta_{\alpha,\beta}$ in all models of $T$. It suffices to show that $\phi$ transfers across elementary extensions. Let $\mc{A}\preceq\mc{B}$, $\ol{a}\in\mc{A}$. If $\mc{A}\models \phi(\ol{a})$, then for some $\alpha$, and every $\beta$, $\mc{A}\models \theta_{\alpha,\beta}(\ol{a})$. Then, for the same $\alpha$, and every $\beta$, $\mc{B}\models \theta_{\alpha,\beta}(\ol{a})$, so $\mc{B}\models \phi(\ol{a})$. Conversely, if $\mc{B}\models\phi(\ol{a})$, then for some $\alpha$, and every $\beta$, $\mc{B}\models\theta_{\alpha,\beta}(\ol{a})$, in which case, for the same $\alpha$ and every $\beta$, $\mc{A}\models \theta_{\alpha,\beta}(\ol{a})$. We conclude that $\mc{A}\models\phi(\ol{a})$.
		
		Suppose that $\psi$ transfers across elementary extensions of models of $T$. We will show that $\psi$ is equivalent to $\mathrm{Force}_\psi$ in every model of $T$. Let $\mathbb{A}$ be a fragment containing $\psi$. If $\mc{A}\models T+ \psi(\ol{a})$, then because $\psi$ transfers across elementary extensions of models of $T$, for any $\mathbb{A}$-generic $\mc{G}\succeq\mc{A}$, $\mc{G}\models\psi(\ol{a})$. Consequently, $\mc{A}\forces^* \psi(\ol{a})$, so $\mc{A}\models \mathrm{Force}_\psi(\ol{a})$. Suppose conversely that $\mc{A}\models T+\mathrm{Force}_\psi(\ol{a})$. Then, $\mc{A}\forces^*\psi(\ol{a})$. Let $\mc{G}\succeq\mc{A}$ be $\mathbb{A}$-generic. Then, $\mc{G}\models T+\psi(\ol{a})$, so because $\psi$ transfers across elementary extensions of models of $T$, $\mc{A}\models\psi(\ol{a})$.
		
	\end{proof}
	
	\subsubsection{\texorpdfstring{$\mc{L}_{\omega_1,\omega}$ and the Malitz Interpolation Theorem}{L{omega1,omega} and the Malitz Interpolation Theorem}}
	
	Note that in Theorem \ref{theorem:preservation}, even if the formula $\psi(\ol{x})$ is in $\mc{L}_{\omega_1,\omega}$, the resulting formula in (2) may not be in $\mc{L}_{\omega_1,\omega}$ (as in Corollary \ref{cor:not-countable}). For the language $\mc{L}_{\omega_1,\omega}$, one can obtain a better result. First, we show what one can get from the Malitz interpolation theorem.
	
	\begin{theorem}[Malitz interpolation theorem \cite{Malitz69}]\label{theorem:Malitz-interpolation}
		Suppose the signature $\tau$ has no function symbols. Let $\varphi$, $\psi$ be sentence of $\mc{L}_{\omega_1,\omega}$ such that $\psi$ is universal $\left(\forall_1\right)$, and $\varphi\models\psi$. Then, there is a universal sentence $\theta$ of $\mc{L}_{\omega_1,\omega}$ such that $\varphi\models\theta$, $\theta\models \psi$ and every symbol occurring in $\theta$ occurs in both $\varphi$ and $\psi$.
	\end{theorem}
	As a consequence of this, Malitz proves the following, which applies to any signature $\tau$. (The version below, in which everything happens relative to a background theory represented by $\sigma$, appears in \cite{Keisler71}; Malitz also shows that if a formula is preserved both upwards and downwards then it is equivalent to a quantifier-free sentence, but this is not true relative to a sentence $\sigma$.)
	\begin{theorem}[Malitz \cite{Malitz69}]\label{theorem:Malitz-preservation}
		Let $\varphi$ and $\sigma$ be sentences of $\mc{L}_{\omega_1,\omega}$. The following are equivalent.
		\begin{enumerate}
			\item If $\mc{A}\subset\mc{B}$, $\mc{A}\models\sigma$, $\mc{B}\models\sigma$, and $\mc{A}\models \varphi$, then $\mc{B}\models\varphi$.
			\item There is an existential sentence $\theta$ of $\mc{L}_{\omega_1,\omega}$ such that $\sigma\models\varphi\leftrightarrow\theta$.
		\end{enumerate}
	\end{theorem}
	
	Note that these theorems are valid only for $\mc{L}_{\omega_1,\omega}$ (and $\mc{L}_{\omega,\omega})$. Malitz \cite{Malitz71} has shown that the Craig interpolation theorem fails in $\mc{L}_{\kappa,\omega}$ for $\kappa > \omega_1$, and indeed there are examples with no interpolant in $\mc{L}_{\infty,\omega}$. We are not sure to what degree Theorem \ref{theorem:Malitz-preservation} fails in $\mc{L}_{\kappa,\omega}$, but Malitz \cite{Malitz69} has shown that there is a set of $\mc{L}_{\omega_1, \omega}$ sentences closed under substructures which is not equivalent to any set of universal $\mc{L}_{\omega_1,\omega}$ sentences. This set of sentences is, however, equivalent to a universal $\mc{L}_{\omega_2,\omega}$ sentence.
	
	Using this theorem, we can give a different characterization of formulas $\mc{L}_{\omega_1,\omega}$ that are preserved by elementary extensions. We say that a formula is $\exists_1$ over finitary formulas if it can be obtained from finitary formulas by taking conjunctions, disjunctions, and existential quantification. Similarly, we say that a formula is $\forall_1$ over finitary formulas if it can be obtained from finitary formulas by taking conjunctions, disjunctions, and universal quantification.
	
	\begin{theorem}\label{theorem:Malitz-preservation-consequence}
		Suppose $\psi$ is a formula of $\mc{L}_{\omega_1,\omega}$. The following are equivalent.
		\begin{enumerate}
			\item  Given $\mc{A} \preceq \mc{B}$, $\mc{A} \models \varphi(\ol{a})$  if and only if $\mc{B} \models \varphi(\ol{a})$.
			\item There are formulas $\alpha$ and $\beta$ of $\mc{L}_{\omega_1,\omega}$ such that $\alpha$ is $\forall_1$ over finitary formulas, $\beta$ is $\exists_1$ over finitary formulas, and $\psi$, $\alpha$, and $\beta$ are all equivalent.
		\end{enumerate}
	\end{theorem}
	
	\begin{proof}[Proof Sketch]
		We omit the full proof as it is straightforward and this theorem will be subsumed by Theorem \ref{theorem:preservationctble} to follow. Essentially one expands the signature by introducing a new relation symbol for each finitary formula and applies Theorem \ref{theorem:Malitz-preservation}.
	\end{proof}
	
	Theorem \ref{theorem:Malitz-preservation-consequence} has the advantage that $\alpha$ and $\beta$ are formulas of $\mc{L}_{\omega_1,\omega}$, while Theorem \ref{theorem:preservation} has the advantage of giving us a single formula, all of whose quantifiers occur in finitary subformulas with complexity bounded by that of $\psi$. We can for the most part combine these advantages, in the following theorem. We say that a formula is quantifier-free over finitary formulas if it can be obtained from finitary formulas by taking conjunctions, disjunctions, and negation.
	
	\preservationtheoremctble*
	
	Recall that the formula that $\psi(\ol{x})$ is equivalent to might involve both $\exists_n$ and $\forall_n$ formulas. We do not know if this is neccesary.
	
	\begin{proof}
		It is a straightforward induction that if $\phi$ is quantifier-free over finitary formulas (2), then $\phi$ transfers across elementary extensions (1).
		
		Suppose $\psi$ transfers across elementary extensions of models of $T$ (1). Restricting the signature to symbols occurring in $\psi$ and $T$, and adding constants for the free variables of $\psi$, it suffices to consider the case where the signature $\tau$ is countable and $\psi$ is a sentence. Suppose that $\psi$ is $\exists_n$. We say that $\mc{A}\equiv_n\mc{B}$ if $\mc{A}$ and $\mc{B}$ satisfy the same \textit{finitary} $\exists_n$ sentences (and hence the same $\forall_n$ sentences).
		
		First we show that for $\mc{A}\models T$, whether $\mc{A} \models \psi$ depends only on the finitary $n$-theory of $\mc{A}$; that is, if $\mc{A},\mc{B}$ are models of $T$ and $\mc{A} \equiv_n \mc{B}$, then $\mc{A} \models \psi$ if and only if $\mc{B} \models \psi$. By Theorem \ref{theorem:preservation}, $\psi$ is equivalent in all models of $T$ to \[\mathrm{Force}_{\psi} = \bigdoublevee\limits_\alpha\bigdoublewedge\limits_\beta \theta_{\alpha,\beta}\] where each $\theta_{\alpha,\beta}$ is a finitary $\exists_n$ sentence.  If $\mc{A}\equiv_n\mc{B}$, then $\mc{A}\models \mathrm{Force}_\psi$ if and only if $\mc{B}\models\mathrm{Force}_\psi$, and so $\mc{A}\models\psi$ if and only if $\mc{B}\models\psi$.
		
		Let $D$ be the set of (finitary) $\exists_n$ sentences of $\mc{L}_{\omega,\omega}$. Given a set $S \subseteq D$, we identify $S$ with the infinitary formula 
		\[  \xi_S = \bigdoublewedge_{\varphi \in S} \varphi \wedge \bigdoublewedge_{\varphi \notin S} \neg \varphi.\]
		Consider the set 
		\[X_\psi = \left\{S\subseteq D \;|\; \text{$T+\xi_S$ is satisfiable and  $T+\xi_S\models \psi$}\right\}\subset 2^D\]
		$D$ is countably infinite, so we can identify $2^D$ with Cantor space, with subbasic clopen sets $[\theta] = \{S\subseteq D\;|\;\theta\in S\}$ and $[\neg \theta] = \{S\subseteq D\;|\;\theta\notin S\}$, for $\theta\in D$. We will show that $X_\psi$ is a Borel set by showing that it is $\bfSigma_1^1$ and $\bfPi_1^1$. If there is a countable model $\mc{A}\models T+\xi_S\wedge \psi$, then $T+\xi_S$ is satisfiable. For any $\mc{B}\models T+\xi_S$, $\mc{B}\equiv_n\mc{A}$, so $\mc{B}\models \psi$, by the above considerations. Thus, $T+\xi_S\models \psi$, so $S\in X_\psi$. Conversely, if $S\in X_\psi$, the L\"owenheim-Skolem theorem for $\mc{L}_{\omega_1,\omega}$ implies that $T+\xi_S\wedge \psi$ has a countable model. Thus
		\[X_\psi = \left\{ S\subseteq D\;|\; \text{$T+\xi_S\wedge \psi$ has a countable model}\right\} \]
		which is $\bfSigma^1_1$. On the other hand, the L\"owenheim-Skolem theorem implies that if $T+\xi_S\nmodels \psi$, there is a countable model $\mc{A}\models T+\xi_S\wedge\neg\psi$. Consequently,
		\[X_\psi = \left\{S\subseteq D\;|\;\text{$T+\xi_S \wedge\neg\psi$ does not have a countable model }\right\} \cap \left\{S\subseteq D\;|\;\text{$T+\xi_S$ is satisfiable}\right\}\] 
		which is the intersection of a $\bfPi^1_1$ set with a Borel set, so is $\bfPi_1^1$.
		
		We can assign to each Borel set $Y\subseteq 2^{D}$ an $\mc{L}_{\omega_1,\omega}$ sentence $\phi_Y$
		which is quantifier-free over finitary $\exists_n$/$\forall_n$ formulas, and equivalent to the $\mc{L}_{(2^\omega)^+,\omega}$ sentence $\bigdoublevee\limits_{S\in Y} \xi_S$. 
		\begin{itemize}
			\item If $Y = [\theta]$ we can take $\phi_Y = \theta$, and if $Y = [\neg \theta]$ we can take $\phi_Y = \neg \theta$.
			\item If $Y = Z^C$, we can take $\phi_Y = \neg\phi_Z$. (This works because $\xi_{S}$ and $\xi_{S'}$ are always inconsistent for $S \neq S'$.)
			\item If $Y = \bigcup\limits_n Z_n$, we can take $\phi_Y = \bigdoublevee\limits_n \phi_{Z_n}$. 
		\end{itemize}
		Any Borel set can be built starting from subbasic clopen sets using countable unions and complements, so we can construct a $\phi_{Y}$ for any $Y$.
		
		We will now show that $\psi$ is equivalent to $\phi_{X_\psi}$ in all models of $T$. If $\mc{A}\models T + \phi_{X_\psi}$, then for some $S\in X_\psi$, $\mc{A}\models \xi_S$. Because $T+\xi_S \models \psi$, $\mc{A}\models \psi$. Conversely, suppose $\mc{A}\models T+\psi$. Let $S = \{\theta \in D|\mc{A}\models\theta\}$ so that  $\mc{A}\models \xi_S$. If $\mc{B}\models T+\xi_S$, then $\mc{A}\equiv_n\mc{B}$, so $\mc{B}\models\psi$. Therefore $T+\xi_S \models \psi$ and so $S \in X_\psi$. Because $\mc{A}\models \xi_S$, $\mc{A}\models \phi_{X_\psi}$.
		
		Thus $\psi$ is equivalent in all models of $T$ to the formula $\phi_{X_\psi}$ which is quantifier-free over finitary $\exists_n$/$\forall_n$ formulas. If we do not assume $\psi$ to be an infinitary $\exists_n$ or $\forall_n$ formula, we have that if $\mc{A}\equiv\mc{B}$, $\mc{A}\models\psi$ if and only if $\mc{B}\models\psi$. Replacing $D$ with all of $\mc{L}_{\omega,\omega}$, the rest of the argument goes through as before. In this case, $\phi_{X_\psi}$ is just quantifier-free over finitary formulas.
	\end{proof}
	
	\bibliography{References}
	\bibliographystyle{alpha}
	
\end{document}